\documentclass[11pt]{amsart}

\usepackage[latin1]{inputenc}
\usepackage{amsmath,amsfonts,amssymb}
\usepackage{amsmath}
\usepackage{amssymb}
\usepackage{amsthm}
\usepackage{latexsym}
\def\opn#1#2{\def#1{\operatorname{#2}}} 
\opn\chara{char} \opn\length{\ell} \opn\pd{pd} \opn\rk{rk}
\opn\projdim{proj\,dim} \opn\injdim{inj\,dim} \opn\rank{rank}
\opn\depth{depth} \opn\codepth{codepth} \opn\grade{grade}
\opn\height{height} \opn\embdim{emb\,dim} \opn\codim{codim}

\opn\Tr{Tr} \opn\bigrank{big\,rank}
\opn\superheight{superheight}\opn\lcm{lcm}
\opn\trdeg{tr\,deg}%
\opn\reg{reg} \opn\lreg{lreg} \opn\skel{skel} \opn\Gr{Gr}
\opn\dim{dim} \opn\arithdeg{arithdeg}

\opn\div{div} \opn\Div{Div} \opn\cl{cl} \opn\Cl{Cl}
\opn\Spec{Spec} \opn\Supp{Supp} \opn\supp{supp} \opn\Sing{Sing}
\opn\Ass{Ass}
\opn\Ann{Ann} \opn\Rad{Rad} \opn\Soc{Soc}
\opn\Sym{Sym} \opn\Ker{Ker} \opn\Coker{Coker} \opn\Im{Im}
\opn\Hom{Hom} \opn\Tor{Tor} \opn\Ext{Ext} \opn\End{End}
\opn\Aut{Aut} \opn\id{id} \opn\ini{in} \opn\tr{tr}

\opn\nat{nat}\opn\it{it}
\opn\pff{proof}
\opn\Pf{proof} \opn\GL{GL} \opn\SL{SL} \opn\mod{mod} \opn\ord{ord}
\opn\aff{aff} \opn\con{conv} \opn\relint{relint} \opn\st{st}
\opn\lk{lk} \opn\cn{cn} \opn\core{core} \opn\vol{vol}
\opn\link{link} \opn\star{star} \opn\skel{skel}
\opn\gr{gr}

\def\pot#1#2{#1[\kern-0.28ex[#2]\kern-0.28ex]}

\opn\dirlim{\underrightarrow{\lim}}
\opn\inivlim{\underleftarrow{\lim}}
\def\Implies{\ifmmode\Longrightarrow \else
     \unskip${}\Longrightarrow{}$\ignorespaces\fi}
\def\implies{\ifmmode\Rightarrow \else
     \unskip${}\Rightarrow{}$\ignorespaces\fi}
\def\iff{\ifmmode\Longleftrightarrow \else
     \unskip${}\Longleftrightarrow{}$\ignorespaces\fi}

\let\:=\colon

\newtheorem{thm}{Theorem}[section]
\newtheorem{cor}[thm]{Corollary}

\theoremstyle{definition}
\newtheorem{defn}[thm]{Definition}
\theoremstyle{remark}
\newtheorem{rem}[thm]{Remark}

\numberwithin{equation}{section}

\let\epsilon\varepsilon
\let\phi=\varphi
\let\kappa=\varkappa
\def\qed{\ifhmode\textqed\fi
   \ifmmode\ifinner\quad\qedsymbol\else\dispqed\fi\fi}
\def\textqed{\unskip\nobreak\penalty50
    \hskip2em\hbox{}\nobreak\hfil\qedsymbol
    \parfillskip=0pt \finalhyphendemerits=0}
\def\dispqed{\rlap{\qquad\qedsymbol}}

\opn\Gin{Gin}

\def\FF{{\mathcal F}}

\opn\inii{in} \opn\inim{inm} \opn\rate{rate}
\begin{document}

\title[Closed graphs are proper interval graphs]
 {Closed graphs are proper interval graphs}


\author[Marilena Crupi]{Marilena Crupi}
\address[Marilena Crupi]{Dipartimento di Matematica e Informatica,
Universit\`a di Messina\\ Viale Ferdinando Stagno d'Alcontres, 31\\ 98166 Messina, Italy.}
\email{mcrupi@unime.it}

\author[Giancarlo Rinaldo]{Giancarlo Rinaldo}
\address[Giancarlo Rinaldo]{Dipartimento di Matematica e Informatica\\
Universit\`a di Messina\\ Viale Ferdinando Stagno d'Alcontres, 31\\ 98166
Messina, Italy}
\email{giancarlo.rinaldo@tiscali.it}





\begin{abstract}
In this note we prove that every closed graph $G$ is up to isomorphism a proper interval graph. As a consequence we obtain that there exist linear-time algorithms
for closed graph recognition.
\end{abstract}

\maketitle

\section*{Introduction}
In this paper a graph $G$ means a simple graph without isolated vertices, loops and multiple edges. Let $V(G)=[n] =\{1, \ldots, n\}$ denote the set of vertices and $E(G)$ the edge set of $G$.

Let $S = K[x_1,\cdots, x_n, y_1,\cdots, y_n]$ be the polynomial ring in $2n$ variables with coefficients in a field $K$. For $i < j$, set $f_{ij} = x_iy_j - x_jy_i$. The ideal $J_G$ of $S$ generated by the binomials $f_{ij} = x_iy_j - x_jy_i$ such that $i<j$ and $\{i,j\}$ is an edge of $G$, is called \textit{the binomial edge ideal} of $G$.  If  $\prec$ is a monomial order on $S$, then a graph $G$ on the vertex set $[n]$ is {\em closed} with respect to the given labeling of the vertices if the generators $f_{ij}$ of $J_G$ form a quadratic Gr\"obner basis (\cite{HH}, \cite{CR}). A combinatorial descrption of this fact is the following. A graph $G$ is closed with respect to the given labeling of the vertices if the following condition
is satisfied: \textit{for all edges $\{i,j\}$ and $\{k, \ell\}$ with $i<j$ and $k<\ell$ one has $\{j,\ell\}\in E(G)$ if $i=k$, and $\{i,k\}\in E(G)$ if $j=\ell$}.

In the last years different authors (\cite{HH}, \cite{HEH}, \cite{CR}, \cite{KM}) have focalized their attention on the class of closed graphs \cite{HH}.

An algorithm that wants to recognize the closedness of a graph $G$ must first of all consider all its $n!$ labelings. But this is not sufficient. In fact, it has to test all the $S$-pairs of the generators of $J_G$, for each labeling of $G$.

Therefore it is useful to know if there exists a linear-time algorithm for closed graph recognition as we have already underlined in \cite{CR}.  In the research of such algorithms we have observed that there exists a bijection between the class of closed graphs and the class of proper interval graphs.

Proper interval graphs are the intersection graphs of intervals of the real line
where no interval properly contains another and have been extensively studied since their inception (\cite{MCG}, \cite{GH}). There are several representations and many characterizations of them (\cite{FG}, \cite{HMP}, \cite{LO}) and some of them through vertex orderings. Such class of  graphs has many applications, such as physical mapping of DNA and genome reconstruction \cite{WG}, \cite{GGKS}.

During the last decade, many linear-time recognition algorithms for proper interval graphs have been developed (\cite{DC}, \cite{DM}, \cite{HHJ}, \cite{PD}) and most of them are based on special breadth-first search (BFS) strategies.

The first linear-time algorithm
for interval graph recognition appeared in 1976 (\cite{BL}). This algorithm uses a \textit{lexicographic breadth first search} (lexBFS) to find in linear time the maximal cliques of the graphs and then employs special structure called $PQ$-trees to find an ordering of the maximal cliques that characterizes interval graphs.  A lexBFS is a breadth first search procedure with the additional rule that vertices with earlier visited neighbors are preferred and its vantage is that it can be performed in $O(\vert V(G)\vert + \vert E(G)\vert)$ time (\cite{RLT}). 
%
%

In this paper we investigate the relation between the closed graphs and the proper interval graphs.

The paper is organized as follows.

Section 1 contains some preliminaries and notions that we will use in the paper.

In Section 2, we state the main results (Theorem \ref{main1}, \ref{main2}) in the paper: \textit{every closed graph is a proper interval graph}. As a consequence we are able to state that by an ordering on the vertices obtained by a lexBFS research
it is possible to test the closedness of a graph in linear-time.
\section{Preliminaries}\label{sec:pre}
In this section we recall some concepts and a notation on graphs and on simplicial
complexes that we will use in the article.

Let $G$ be a simple graph with vertex set $V(G)$ and the edge set $E(G)$. Let $v, w \in V(G)$. A path $\pi$ from $v$ to $w$ is a sequence of vertices $v=v_0, v_1, \cdots, v_t=w$ such that $\{v_i, v_{i+1}\}$ is an edge of  the underlying graph. A graph $G$ is \textit{connected} if for every pair of vertices $v_1$ and $v_2$ there is a path from $v_1$ to $v_2$. 


When we fix a given labelling on the vertices we say that $G$ is a graph on $[n]$.

Let $G$ be a graph with vertex set $[n]$. A subset $C$ of $[n]$ is called a \textit{clique} of $G$ is for all $i$ and $j$ belonging to $C$ with $i \neq j$ one has $\{i, j\} \in E(G)$.

Set $V = \{x_1, \ldots, x_n\}$. A \textit{simplicial complex}
$\Delta$ on the vertex set $V$ is a collection of subsets of $V$
such that
\begin{enumerate}
\item[(i)] $\{x_i\} \in \Delta$  for all $x_i \in V$ and
\item[(ii)] $F \in \Delta$ and $G\subseteq F$ imply $G \in \Delta$.
\end{enumerate}
An element $F \in \Delta$ is called a \textit{face} of $\Delta$. For
$F \in \Delta$ we define the \textit{dimension} of $F$ by $\dim F
= |F| -1$, where $|F|$ is the cardinality of the set
$F$. A maximal face of $\Delta$  with respect to inclusion is
called a \textit{facet} of $\Delta$.

If $\Delta$ is a simplicial complex  with facets $F_1, \ldots, F_q$, we call $\{F_1, \ldots, F_q\}$ the facet set of $\Delta$ and we denote it by $\FF(\Delta)$. When $\FF(\Delta)=\{F_1, \ldots, F_q\}$, we write $\Delta=\langle F_1, \ldots, F_q\rangle$.

\begin{defn}\label{def:cli}
The \textit{clique complex} $\Delta(G)$ of $G$ is the simplicial complex whose faces are the cliques of $G$.
\end{defn}

The clique complex plays an important role in the study of the class of \textit{closed graphs} (\cite{HH}, \cite{CR}).

\begin{defn} A graph $G$ is closed with respect to the given labeling if the following condition
is satisfied:

For all edges $\{i,j\}$ and $\{k, \ell\}$ with $i<j$ and $k<\ell$ one has $\{j,\ell\}\in E(G)$ if $i=k$, and $\{i,k\}\in E(G)$ if $j=\ell$.
\end{defn}
In particular, $G$ is closed if there exists a labeling for which it is closed.

Since a  graph is closed if and only if each connected component is closed we assume from now on that the graph $G$ is connected.
\begin{thm} Let $G$ be a graph. The following conditions are equivalent:
\begin{enumerate}
 \item[(1)] there exists a labelling $[n]$ of $G$ such that $G$ is closed on $[n]$;
 \item[(2)] $J_G$ has a quadratic Gr\"obner basis with respect to some term order $\prec$ on $S$;
 \item[(3)] there exists a labelling of $G$ such that all facets of $\Delta(G)$ are intervals $[a,b] \subseteq [n]$.
\end{enumerate}
\end{thm}
\begin{proof} (1) $\Leftrightarrow $ (2): see \cite{CR}, Theorem 3.4.

(1) $\Leftrightarrow $ (3): see \cite{HEH}, Theorem 2.2.
\end{proof}


\section{Closed graphs are proper interval graphs}\label{sec:closed}
In this section we want to underline the relation between the class of the closed graphs and the class of proper interval graphs.

\begin{defn} An graph $G$
is an interval graph if to each vertex $v \in V(G)$ a closed interval $I_v = [\ell_v, r_v]$ of the real line
can be associated, such that two distinct vertices $u, v \in V(G)$ are adjacent if
and only if $I_u \cap I_v \neq \emptyset$.
\end{defn}

The family $\{I_v\}_{v\in V(G)}$ is \textit{an interval representation}
of $G$.

A graph $G$ is a \textit{proper interval graph} if there is an interval
representation of $G$ in which no interval properly contains another. In the
same way, a graph G is a \textit{unit interval graph} if there is an interval
representation of $G$ in which all the intervals have the same length.

If $G$ is a graph, a \textit{vertex ordering} $\sigma$ for $G$ is a permutation of $V(G)$.
We write $u \prec_{\sigma} v$ if $u$ appears before $v$ in $\sigma$. Ordering $\sigma$ is called a \textit{proper
interval ordering} if for every triple $u, v,w$ of vertices of $G$ where $u \prec_{\sigma} v \prec_{\sigma} w$
and $\{u, w\}\in E(G)$, one has $\{u,v\}, \{v,w\} \in E(G)$. We call this condition the \textit{umbrella property}.

The vertex orderings allows to state many characterizations of proper interval graphs. We quote the next result from \cite{LO}, Theorem 2.1.
\begin{thm} \label{vertex} A graph $G$ is a proper interval graph if and only if $G$
has a proper interval ordering.
\end{thm}

The next result shows the relation between a closed graph and a proper interval graph.

\begin{thm} \label{main1} Let $G$ be a closed graph. Then there exists a proper interval graph $H$ such that $G \simeq H$.
\end{thm}
\begin{proof} Since $G$ is closed then there exists a labelling $[n]$ of $G$ such that all facets of the clique complex $\Delta(G)$ are intervals $[a,b] \subseteq [n]$, that is
\begin{equation}\label{inter1}
    \Delta(G) =
\langle [a_1,b_1], [a_2,b_2],\ldots,[a_r,b_r]\rangle,
\end{equation}
with $1=a_1<a_2<\ldots<a_r<n$, $1<b_1<b_2<\ldots<b_r=n$  with $a_i < b_i$ and $a_{i+1} \leq b_i$, for $i \in [r]$.

Set $\varepsilon = \displaystyle{\frac 1 n}$. Define the following closed intervals of the real line:
\[I_k = [k, b(k) + k \varepsilon],\]
where
\begin{equation}\label{max}
    \mbox{$b(k)= \max\{b_i\,:\, k \in [a_i,b_i]\},$\qquad  for $k =1, \ldots, n$}.
\end{equation}

Let $H$ be the interval graph on the set $V(H) = \{I_1, \ldots, I_n\}$ and let
\[\varphi: V(G)=[n] \rightarrow V(H)\]
be the map defined as follows:
\[\varphi(k) = I_k.\]
{\bf Claim:} $\varphi$ is an isomorphism of graphs.

Let $\{k, \ell\} \in E(G)$ with $k < \ell$. We will show that $\{\varphi(k),  \varphi(\ell)\} = \{I_k, I_{\ell}\} \in E(H)$, that is, $I_k\cap I_{\ell} \neq \emptyset$.

It is
\[I_k = [k, b(k) + k \varepsilon],\qquad I_{\ell} = [\ell, b(\ell) + \ell \varepsilon].\]

Suppose $I_k\cap I_{\ell} = \emptyset$. Then $b(k) + k \varepsilon < \ell$ and consequently $b(k) < \ell$. It follows that does not exist a clique containing the edge $\{k, \ell\}$. A contradiction.

Suppose that $\{I_k, I_{\ell}\} \in E(H)$, with $k < \ell$. We will prove that $\{k, \ell\} \in E(G)$.

Since $I_k\cap I_{\ell} \neq \emptyset$, then $b(k) + k \varepsilon \geq \ell$. By the meaning of $\varepsilon$ and by the assumption $k < \ell$, it follows that $k\varepsilon < 1$ and so $b(k) \geq \ell$. Hence from (\ref{inter1}) and (\ref{max}),  $\{k, \ell\} \in E(G)$.

Since $G$ is closed and consequently a $K_{1,3}$-free graph \cite{FSR}, the isomorphism $\varphi$ assures that $H$ is a proper interval graph.
\end{proof}
%
The next result shows the relation between a proper interval graph and a closed graph.

\begin{thm} \label{main2}Let $G$ be a proper interval graph. Then there exists a closed graph $H$ such that $G \simeq H$.
\end{thm}
\begin{proof} Let $G$ be a proper interval graph and let  $\{I_v\}_{v\in V(G)}$ an \textit{interval representation}
of $G$, with $\vert V(G)\vert = n$.

From Theorem \ref{vertex}, there exists a proper interval ordering $\sigma$ of $G$. Let $\sigma=(I_1,\ldots, I_n)$ of $G$ be such vertex ordering.
It is  $I_j\prec_{\sigma} I_k$ if and only if $j <k$.

Let $H$ be the graph with vertex set $V(H) = [n]$ and edge set $E(H) = \{\{i,j\}\,:\, \{I_i,I_j\}\in E(G)\}$.

{\bf Claim:} $H$ is a closed graph on $[n]$.

Let $\{i,j\}, \{k, \ell\}\in E(H)$ with $i<j$ and $k<\ell$.

Suppose $i=k$. Since $\{i,j\}, \{i, \ell\}\in E(H)$, then $\{I_i,I_j\}, \{I_i,I_{\ell}\}\in E(G)$.

If $i < j < \ell$, then $I_i \prec_{\sigma}I_j \prec_{\sigma} I_k$. Hence since $\sigma$ satisfies the ummbrella property and $\{I_i,I_{\ell}\}\in E(G)$, it follows that $\{I_i,I_j\}, \{I_j,I_{\ell}\}\in E(G)$. Thus $\{j,\ell\}\in E(H)$.

Repeating the same reasoning for $i < \ell < j$, it follows that $\{j,\ell\}\in E(H)$ again.

Similarly for $j=\ell$, one has $\{i, k\}\in E(H)$. Hence $H$ is a closed graph.

It is easy to verify that the proper interval graph $G$  is isomorphic to the closed graph $H$ by the map $\psi: V(G) \rightarrow V(H)=[n]$, that sends every closed interval $I_j\in V(G)$ to the integer $j\in V(H)$.
\end{proof}
%

As a consequence of Theorems \ref{main1}, \ref{main2} we have that:
\begin{cor} \label{equiv} For an undirected graph $G$, the following statements are equivalent:
\begin{enumerate}
\item $G$ is a closed graph;
\item $G$ is a proper interval graph;
\item the clique-vertex incidence matrix of $G$ has the consecutive 1s property both for rows and for columns;
\item $G$ is a unit interval graph;
\item $G$ is a $K_{1,3}$-free interval graph.
\end{enumerate}
\end{cor}
\begin{proof} (1)$\Leftrightarrow$ (2): Theorems \ref{main1}, \ref{main2}.

(2)$\Leftrightarrow$ (3) $\Leftrightarrow$ (4) $\Leftrightarrow$ (5): \cite{FG}, Theorem 1.
\end{proof}

\begin{rem} Let $G$ be a graph on the vertex set $[n]$. If we choose $\sigma = id_{[n]}$ as vertex ordering for $G$, then $i \prec_{\sigma} j$ if and only if $i <j$, for every pair $i,j \in [n]$. Then the umbrella property for $\sigma = id_{[n]}$ can be rewritten as follows:

\textit{for every triple $u, v,w$ of vertices of $G$ with $u < v < w$
and $\{u, w\}\in E(G)$, one has $\{u,v\}, \{v,w\} \in E(G)$.}

Hence, from Theorem \ref{vertex} and Corollary \ref{equiv}, it follows that $G$ is a closed graph on the vertex set $[n]$ if and only if $\sigma = id_{[n]}$ satisfies the umbrella property. See also \cite{KM}.
\end{rem}


\bibliographystyle{plain}

\end{document}